\documentclass[12pt]{amsart}
\usepackage{amsmath, amsfonts, upref, amsbsy, appendix, amssymb, amsthm, setspace}

\newtheorem{thm}{Theorem}[section]
\newtheorem{lmm}[thm]{Lemma}
\newtheorem{cor}[thm]{Corollary}
\newtheorem{prop}[thm]{Proposition}

\newcommand{\ee}{\mathbb{E}}

\newcommand{\mf}{\mathcal{F}}

\newcommand{\var}{\mathrm{Var}}

\begin{document}
\title{Random multiplicative functions in short intervals}
\author{Sourav Chatterjee}
\address{Courant Institute of Mathematical Sciences, New York University, 251 Mercer Street, New York, NY 10012}
\email{sourav@cims.nyu.edu}
\thanks{Sourav Chatterjee's research was partially supported by  NSF grant DMS-1005312   and a Sloan Research Fellowship}
\author{Kannan Soundararajan}
\address{Department of Mathematics, Stanford University, Stanford, CA 94305.}
\email{ksound@stanford.edu}
\thanks{Kannan Soundararajan's research was supported in part by NSF grant DMS-1001068}  
\begin{abstract}
We consider random multiplicative functions taking the values $\pm 1$.  Using Stein's method for normal approximation, we prove a central limit theorem for the sum of such multiplicative functions in appropriate short intervals. 
\end{abstract}

\maketitle

\section{Introduction} 

Many of the functions of interest to number theorists are multiplicative.  That is 
they satisfy $f(mn)= f(m) f(n)$ for all coprime natural numbers $m$ and $n$.   Some 
examples are the M{\" o}bius function $\mu(n)$, the function $n^{it}$ for a real number $t$, 
and Dirichlet characters $\chi(n)$.  Often one is interested in the behavior of partial sums 
$\sum_{n\le x} f(n)$ of such multiplicative functions. For the proto-typical 
examples mentioned above it is a difficult problem to obtain a good understanding of 
such partial sums.  A guiding principle that has emerged is that partial sums of 
specific multiplicative functions (e.g.\ characters or the M{\" o}bius function) behave 
like partial sums of random multiplicative functions.   By random we mean that 
the values of the multiplicative function at primes are chosen randomly, and the 
values at all natural numbers are built out of the values at primes by the 
multiplicative property.  For example this viewpoint is explored in the 
context of finding large character sums in \cite{GS}.   

This raises the question of the distribution of partial sums of random multiplicative functions, and even this 
model problem appears difficult to resolve.   The aim of this paper is to study the distribution of random multiplicative functions in  
short intervals $[x, x+y]$, and in suitable ranges we shall establish 
that the sum of a random multiplicative function in that range has an approximately Gaussian distribution.  

Throughout $p$ will denote a prime number, and let $X(p)$ denote independent 
random variables taking the values $+1$ or $-1$ with equal probability.  
Let $X(n)=0$ if $n$ is divisible by the square of any prime, and if $n= p_1 \cdots p_k$ is 
square-free we define $X(n) = \prod_{j=1}^{k} X(p_j)$.  Let $M(x) = \sum_{n\le x} X(n)$.  
In \cite{H}, Halasz showed that  with probability $1$ we have 
$$ 
|M(x)| \le c x^{\frac 12} \exp\Big( d (\log \log x \log \log \log x)^{\frac 12}\Big),
$$ 
for some positive constants $c$ (which may depend on random function $X$) 
and $d$ (an absolute constant), and forthcoming work of 
Lau, Tenenbaum and Wu \cite{LTW} substantially improves 
upon this bound.   Furthermore, Halasz showed that 
 with positive probability the estimate 
 $M(x) \ge c x^{\frac 12} \exp(-d (\log \log x \log \log \log x)^{\frac 12})$ 
 holds infinitely often (for any $d>0$), and this has been substantially improved in forthcoming work of 
 Harper \cite{Ha2}.  These results may be seen as   approximations to the 
 law of the iterated logarithm for sums of independent random variables.   
 In related recent works Hough \cite{Ho} and Harper \cite{Ha} have considered the 
 distribution of $\sum_{n\le x}^{\prime} X(n)$, where the sum is restricted to 
 integers having exactly $k$ prime factors.  Note that the central limit 
 theorem covers the case $k=1$ when we have a sum of independent random variables.  
 When $k$ is a fixed positive integer, using the method of moments Hough established that such sums 
 have a Gaussian distribution.  The work of Harper extends Hough's result and 
  using the martingale central limit theorem he established 
 that the Gaussian distribution persists for $k=o( \log \log x)$, 
 and fails for $k$ of size a constant times $ \log \log x$.  Recall that most numbers $n\le x$ 
 have about $\log \log x$ prime factors, and so the dichotomy seen in Harper's result 
 is quite interesting.  Harper also showed by a conditioning argument that $M(x)$ itself cannot have a 
 normal distribution with mean $0$ and variance the number of square-free integers below 
 $x$.

  \begin{thm}\label{mainthm} Let $X$ denote a random multiplicative function as above.  Let 
  $x$ and $y$ be large natural numbers with $y=\delta x$ for some $\delta <1/10$.  
  Let $S=S(x,y)$ denote the number of square-free integers in $[x,x+y]$.  Let $Z$ 
  denote a Gaussian random variable with mean $0$ and variance $1$, and 
  let $\phi$ denote a Lipschitz function satisfying $|\phi(\alpha)-\phi(\beta)|\le |\alpha-\beta|$ 
  for all real numbers $\alpha$ and $\beta$.  Then we have that
$$
\Big  |{\Bbb E}\Big( \phi\Big( \frac{1}{\sqrt{S} } \sum_{x<n\le x+y}X(n)\Big)\Big) - 
  {\Bbb E} \phi(Z) \Big| 
$$
is bounded by a constant times
$$  
\min \Big(1,  \Big(\frac{y}{S }\Big)^{\frac 32} \frac{1}{ (\log 1/\delta)^{\frac 12}} + \frac yS\sqrt{\delta \log x  }
  + \frac{y\log x}{S^{\frac 32} \log y}\Big).
  $$
    \end{thm}  
  
  We recall that the Kantorovich-Wasserstein distance between 
  two probability measures $\mu$ and $\nu$ on the real 
  line, denoted ${\mathcal W}(\mu,\nu)$, is defined as the supremum of $|\int hd\mu -\int hd\nu|$ over 
  all Lipschitz functions $h$ satisfying $|h(\alpha)-h(\beta)| \le |\alpha-\beta|$ for 
  all real numbers $\alpha$ and $\beta$.   Thus our Theorem gives an estimate 
  for the Kantorovich-Wasserstein distance between a normal distribution with 
  mean zero and variance $1$, and the distribution of sums of random multiplicative 
  functions in short intervals.    An intuitive way to assess 
  the distance between two probability measures is the Kolmogorov statistic: 
  ${\mathcal K}(\mu,\nu) = \sup_{x\in {\Bbb R}} | \int_{-\infty}^{x} d\mu - \int_{-\infty}^{x} d\nu|$.  
  By a standard smoothing argument, we shall show how our estimate 
  for the Kantorovich-Wasserstein distance can be used to bound the Kolmogorov statistic.  

\begin{cor}  With notations as in Theorem \ref{mainthm} we have that 
$$ 
\sup_{t\in {\Bbb R}} \Big| {\Bbb P} \Big( \frac{1}{\sqrt{S}} \sum_{x<n\le x+y} X(n) \in (-\infty, t)\Big) 
- \frac{1}{\sqrt{2\pi}} \int_{-\infty}^{t} e^{-z^2/2} dz \Big| 
$$ 
is bounded by a constant times 
$$ 
\min \Big(1 , \Big(\frac{y}{S}\Big)^{\frac 34} \frac{1}{(\log 1/\delta)^{\frac 14}} 
+ \Big( \frac yS\Big)^{\frac 12} (\delta \log x)^{\frac 14} + 
\frac{\sqrt{y\log x}}{S^{\frac 34} \sqrt{\log y}}\Big).
$$ 
\end{cor}

 In an interval $[x,x+y]$ we expect that there are about $\sim \frac{6}{\pi^2} y$ 
 square-free integers.  The work of Filaseta and Trifonov \cite{FT} shows that 
 if $x\ge y\ge C x^{\frac 15}\log x$ for some positive 
 constant $C$ then a positive proportion 
 of the integers in $[x,x+y]$ are square-free.  The theorem in Filaseta 
 and Trifonov only asserts the existence of a square-free integer in such an 
 interval, but their proof plainly gives the stronger result above.  Therefore 
 for all short intervals with $Cx^{\frac 15}\log x< y = o(x/\log x)$, 
 our Theorem shows that the distribution of $\sum_{x< n \le x+y} X(n)$ 
 is approximately normal.  
 Granville \cite{G} has shown that the $ABC$-conjecture implies that 
 the interval $[x,x+y]$ contains a positive proportion of 
 square-free integers if $x^{\epsilon} \ll y\le x$ for any $\epsilon >0$; again 
 Granville only stated the existence of one square-free integer in such 
 intervals, but his proof gives the stronger assertion above.  Thus, 
 on the $ABC$-conjecture, for any short interval with $x^{\epsilon}\ll 
 y = o(x/\log x )$ our Theorem shows that 
 the distribution of $\sum_{x< n \le x+y} X(n)$ is approximately normal.

 The proof of this result is based on a version of Stein's method for 
 normal approximation developed in \cite{chatterjee08}.    This involves calculating 
 quantities related to the fourth moment of $\sum_{x<n\le x+y} X(n)$.  
 The fourth moment itself is calculated in Proposition 3.1 below.  
 If the interval $[x,x+y]$ contains a positive proportion of 
 square-free numbers, then Proposition 3.1 shows that 
 the fourth moment is asymptotically the fourth moment of a 
 normal distribution provided $y = o(x/\log x)$.  Further, 
when $y$ is of size a constant times $x/\log x$, the argument there 
shows that the fourth moment does not match the fourth moment of a 
normal distribution.  Thus it seems plausible that for $x/\log x \ll y\le x$ 
the distribution of $\sum_{x<n\le x+y} X(n)$ is not normal, but 
we do not have a proof of this assertion.   By modifying the conditioning 
argument in Harper \cite{H} we can establish that if $y$ is of a constant 
times $x$ then the distribution of $\sum_{x< n \le x+y} X(n)$ is not normal.  

The method developed here could also be used to study 
the distribution of $\sum_{n\in {\mathcal S}} X(n)$ for 
other subsets ${\mathcal S}$ of square-free numbers in $[1,x]$.  
For example, we can obtain in this manner a different treatment 
of the results of Harper and Hough.   Another example is the 
set of integers below $x$ that are $\equiv a\pmod q$ where $(a,q)=1$.  
If $q/\log x$ is large, and this arithmetic progression contains the expected number 
of square-free integers, the distribution should be normal analogously to Theorem 1.1.

 \section{Beginning of the proof}  
 
 Let $x$, $y$ and $\delta$ be as in the statement of the 
 Theorem, and let $X$ denote a random multiplicative function as 
 defined in the Introduction.   We let $z$ denote $\frac 12 \log (1/\delta)$.  
 We divide the primes below $2x$ into large (that is $>z$) 
 and small (that is $\le z$) primes.  We denote the 
 set of large primes by ${\mathcal L}$, and the set 
 of small primes by ${\mathcal S}$.   Let ${\mathcal F}$ be 
 the sigma-algebra generated by $X(p)$ for all $p\in {\mathcal S}$, 
 and we denote the conditional expectation given ${\mathcal F}$ 
 by ${\Bbb E}^{\mathcal F}$.  
 
 Let $X_{\mathcal L}$ denote the vector $(X(p))_{p\in {\mathcal L}}$.  
 Then, given ${\mathcal F}$, we may think of $\sum_{x<n\le x+y}X(n)$ 
 as a function of $X_{\mathcal L}$, and we write this function as $f(X_{\mathcal L})$.  
 
 \begin{lmm} With the above notations we 
 have 
 \[ 
 {\Bbb E}^{\mathcal F}(f(X_{\mathcal L}))  = 0, 
 \] 
 and 
 \[ 
 {\Bbb E}^{\mathcal F}(f(X_{\mathcal L})^2) = S(x,y). 
 \] 
 \end{lmm} 
 \begin{proof} Write a square-free number $n \in [x,x+y]$ 
 as $n_{\mathcal S} n_{\mathcal L}$ where $n_{\mathcal S}$ 
 is the product of the primes in ${\mathcal S}$ that divide $n$, and $n_{\mathcal L}$ 
 the product of the primes in ${\mathcal L}$ that divide $n$.  
 From our choice of $z=\frac 12 \log (1/\delta)$ we note that $n_{\mathcal S} \le \prod_{p\le z}p 
 \le 4^z$.  It follows that $n_{\mathcal L} = n/n_{\mathcal S} > \delta x =y$.  
 From this we obtain that ${\Bbb E}^{\mf} (f(X_{\mathcal L})) = 0$.  
 Moreover, note that if $n$ and $n^{\prime}$ are distinct square-free numbers 
 in $[x,x+y]$ then we must have $n_{\mathcal L} \neq n_{\mathcal L}^{\prime}$.  
 Therefore we deduce that ${\Bbb E}^{\mf} (f(X_{\mathcal L})^2) = S(x,y)$, 
 proving our Lemma. 
 \end{proof} 
 \vskip.1in
 Let $X_{\mathcal L}^{\prime}$ denote an independent copy of $X_{\mathcal L}$.  
 For each subset ${\mathcal A}$ of ${\mathcal L}$ we write $X_{\mathcal L}^{\mathcal A}$ to be 
 the vector defined as $X^{\mathcal A}(p) = X(p) $ for $p \in {\mathcal L}\backslash {\mathcal A}$, 
 and $X^{\mathcal A}(p)= X^{\prime}(p)$ for $p\in {\mathcal A}$.   
 For a proper subset ${\mathcal A}$ of ${\mathcal L}$, and a prime $p\in {\mathcal L}\backslash 
 {\mathcal A}$ we define 
 $$ 
 \Delta_pf := f(X_{\mathcal L}) - f(X_{\mathcal L}^{\{p\}}), 
 $$ 
 and 
 $$
 \Delta_pf^{\mathcal A} := f(X_{\mathcal L}^{\mathcal A}) - f(X_{\mathcal L}^{{\mathcal A}\cup 
 \{ p\}}).
 $$ 
 Finally define 
 $$
 T:= \frac 12 \sum_{{\mathcal A} \subsetneq {\mathcal L}} \frac{1}{\binom{|{\mathcal L}|}{|{\mathcal A}|} 
 (|{\mathcal L}|-|{\mathcal A}|)} \sum_{p \in {\mathcal L}\backslash {\mathcal A} } \Delta_p (f) \Delta_p(f^{\mathcal A}).
 $$

With these notations, and Lemma 2.1,  Theorem 2.2 from \cite{chatterjee08} enables us to get the following result.
 
\begin{prop}\label{normthm} 
Let $Z$ denote a random variable with a Gaussian distribution with mean 
zero and variance $1$.  Let $W= \frac{1}{\sqrt{S}} \sum_{x<n \le x+y}X(n)$, 
and let $\phi$ denote a Lipschitz function satisfying $|\phi(\alpha)-\phi(\beta)|\le |\alpha-\beta|$ 
for all real numbers $\alpha$ and $\beta$.  We have
\[
|\ee^\mf\phi(W) - \ee \phi(Z)| \le \frac{(\var^\mf(\ee^\mf(T|X)))^{1/2}}{S} + \frac{1}{2S^{3/2}}\sum_{p\in 
{\mathcal L}} \ee^\mf|\Delta_pf|^3.
\]
\end{prop}
Here conditioning on $X$ means that we are conditioning on the whole vector $(X(n))_{n\ge 1}$. 
Actually, the bound given by Theorem 2.2 from the paper \cite{chatterjee08} has $\var^\mf(\ee^\mf(T|W))$ in the first term instead of $\var^\mf(\ee^\mf(T|X))$. However, the latter quantity is at least as 
large as the former because $\ee^\mf(T|W) = \ee^\mf(\ee^\mf(T|X)|W)$ and conditioning reduces variance.

 We shall use Proposition 2.2 to estimate $|{\Bbb E}(\phi(W)) - {\Bbb E}(\phi(Z))|$.  Note 
 that this quantity is bounded by 
 $$ 
 {\Bbb E} | {\Bbb E}^{\mathcal F}(\phi(W)) - {\Bbb E}(\phi(Z))| 
 \le \frac{1}{S} {\Bbb E} \Big(  \var^\mf(\ee^\mf(T|X)))^{1/2}\Big) 
 + \frac{1}{2S^{\frac 32}} \sum_{p \in {\mathcal L} }  \ee|\Delta_pf|^3.
 $$
 By the Cauchy-Schwarz inequality the first term above is 
 $$ 
 \le \frac 1S \Big( {\Bbb E}  \var^\mf(\ee^\mf(T|X)))\Big)^{\frac 12} \le  
 \frac 1S \Big( \var( {\Bbb E}^{\mf} (T|X))\Big)^{\frac 12}.  
  $$ 
  We deduce that 
  \begin{equation} 
  \label{est1}
  |{\Bbb E}(\phi(W)) - {\Bbb E}(\phi(Z))| \le  \frac 1S \Big( \var( {\Bbb E}^{\mf} (T|X))\Big)^{\frac 12} 
   + \frac{1}{2S^{\frac 32}} \sum_{p\in {\mathcal L}} {\Bbb E}|\Delta_pf|^3. 
   \end{equation} 
   We will now focus on estimating the 
   two terms in the RHS above.  The second term will be estimated in the 
   next section, and the first in Section 4.  We now simplify the expression in the 
   first term a little.  
   
   For each $p\in {\mathcal L}$, let ${\mathcal N}(p)$ denote all square-free numbers in the interval 
   $[x/p, (x+y)/p]$ which are coprime to $p$. Note that
\[
\Delta_p f = (X(p) - X^{\prime}(p))\sum_{k\in {\mathcal N}(p)} X(k), 
\]
and if $p\in {\mathcal L} \backslash {\mathcal A}$, 
\[
\Delta_pf^{\mathcal A} = (X(p) - X^{\prime}(p)) \sum_{k\in {\mathcal N}(p)} X^{\mathcal A}(k),
\]
where $X^{\mathcal A}(k)$ is defined in the obvious way replacing $X$ by $X^{\prime}$ on the 
primes in ${\mathcal A}$. Therefore
\begin{equation*}\label{prod}
\begin{split}
\Delta_p f \Delta_p f^{\mathcal A} &= (X(p)-X'(p))^2\biggl(\sum_{k \in {\mathcal N}(p)} X(k)\biggr)
\biggl(\sum_{ \ell \in {\mathcal N}(p)} X^{\mathcal A}(\ell)\biggr),\\
\end{split}
\end{equation*}
and since $X(k)$ and $X^{\mathcal A}(\ell)$ do not depend on $X_p$, we see that
\begin{align*}
\frac{1}{2}\ee^\mf (\Delta_p f\Delta_p f^{\mathcal A} \mid X) &= \biggl(\sum_{k\in {\mathcal N}(p)} X(k)\biggr)\biggl(\sum_{\ell \in {\mathcal N}^{\mathcal A}(p)} X(\ell)\biggr)\\
&= |{\mathcal N}^{\mathcal A}(p)| + \sum_{{k \in {\mathcal N}(p), \  \ell \in {\mathcal N}^{\mathcal A}(p)} 
\atop {k \neq \ell} } X(k) X(\ell),
\end{align*}
where ${\mathcal N}^{\mathcal A}(p)$ 
denotes the set of all square-free integers in $[x/p,(x+y)/p]$ that are not divisible $p$ and 
by any prime $q\in {\mathcal A}$. 
Write  the quantity $T$ in Proposition~\ref{normthm}  as
\begin{align*}
T&= \frac{1}{2}\sum_{p\in {\mathcal L}} \sum_{{\mathcal A}\subseteq {\mathcal L}\backslash\{p\}} 
\nu({\mathcal A}) \Delta_pf \Delta_pf^{\mathcal A},
\end{align*}
where
\[
\nu({\mathcal A}) := \frac{1}{{|{\mathcal L}| \choose |{\mathcal A}|}(|{\mathcal L}|-|{\mathcal A}|)} 
= \frac{1}{|{\mathcal L}|{|{\mathcal L}|-1\choose |{\mathcal A}|}}. 
\]
Thus, 
\begin{align*}
\ee^\mf(T\mid X) &= \frac{1}{2}\sum_{p\in {\mathcal L}} \sum_{{\mathcal A}\subseteq {\mathcal L}\backslash \{p\}} \nu({\mathcal A}) \ee^\mf (\Delta_p f\Delta_p f^{\mathcal A} \mid X)\\
&= \sum_{p\in {\mathcal L}} \sum_{{\mathcal A}\subseteq {\mathcal L}\backslash \{p\}} \nu({\mathcal A}) |{\mathcal N}^{\mathcal A}(p)| \\
&\qquad + \sum_{p\in {\mathcal L}} \sum_{k\in {\mathcal N}(p)}
\sum_{\ell \in {\mathcal N}(p)\backslash \{k\}} \frac{1}{\omega_{\mathcal L}(\ell p)} X(k) X(\ell), 
\end{align*}
where $\omega_{\mathcal L}(n)$ denotes the number of distinct prime factors of $n$ that are 
in $\mathcal L$ and 
the equality above holds because 
\begin{align*}
 & \sum_{{\mathcal A}\subseteq {\mathcal L}\backslash\{p\}} \nu({\mathcal A}) 1_{\{\ell\in {\mathcal N}^{\mathcal A}(p)\}}
 =\sum_{{\mathcal A}\subseteq {\mathcal L} \backslash (\{p\}\cup \{ q|\ell \})} \nu({\mathcal A}) \\
&= \sum_{k=0}^{|{\mathcal L}|-\omega_{\mathcal L}(\ell p)} \frac{1}{|{\mathcal L}|{|{\mathcal L}|-1\choose k}} {|{\mathcal L}|-\omega_{\mathcal L}(\ell p)\choose k} = \frac{1}{\omega_{\mathcal L}(\ell p)}. 
\end{align*}
The last step above involves a combinatorial identity and we leave the 
pleasure of proving it to the reader; a generalization of this identity appears 
as Problem B2 of the 1987 Putnam competition see \cite{KPV}.  
Now we define 
\[
T_p := \sum_{k\in {\mathcal N}(p)}\sum_{\ell \in {\mathcal N}(p)\backslash \{k\}} \frac{1}{\omega_{\mathcal L}(\ell p)} X(k) X(\ell). 
\]
Then we may conclude that 
\begin{equation}
\label{est2} 
\var({\Bbb E}^{\mf} (T\mid X) ) = \var\Big( \sum_{p \in {\mathcal L}} T_p \Big). 
\end{equation}


\section{The fourth moment and a parametrization of solutions}

In this section we shall evaluate the fourth moment 
$$
{\Bbb E}\Big( \sum_{x<n\le x+y} X(n)\Big)^4, 
$$ 
for a suitable range of the variables $x$ and $y$.  The techniques 
involved in this calculation will be used in the proof of our main 
Theorem.   When we expand out the fourth moment, we find 
that we are counting solutions to the equation 
\[
 n_1 n_2 n_3 n_4 =\square
\]
where $n_1$, $n_2$, $n_3$, $n_4$ are square-free 
integers with $n_j \in [x, x+y]$ and $\square$ denotes a perfect square.  Recall that $y=x\delta$.   
We begin by parametrizing such solutions.  

Write $A=(n_1,n_2)$ and $B=(n_3, n_4)$, and set 
$n_1=A n_1^*$, $n_2=An_2^*$, $n_3=Bn_3^*$ and $n_4=Bn_4^*$.  
Then $(n_1^*,n_2^*)= (n_3^*,n_4^*)=1$ and the equation 
$n_1n_2 n_3n_4= \square$ is equivalent to $n_1^*n_2^*=n_3^*n_4^*$.  
Now write $r=(n_1^*,n_3^*)$ and $s=(n_2^*,n_4^*)$.  Then $(r,s)=1$ 
and we see that $n_1^*=ru$, $n_3^*=rv$, $n_2^*=sv$ and $n_4^*=su$ 
where $u$ and $v$ are natural numbers with $(u,v)=1$.   

Summarizing the above paragraph, we see that the solutions 
to $n_1n_2n_3n_4=\square$ are parametrized by six variables 
$A$, $B$, $r$, $s$, $u$, $v$, with $(r,s)=(u,v)=1$ and 
with 
\[ 
n_1 = Aru, n_2= Asv, n_3=Brv, n_4=Bsu. 
\]
  There are 
additional coprimality conditions to ensure that these numbers 
are square-free.  Since $(1+\delta)^{-2} \le n_1n_2/(n_3n_4) \le (1+\delta)^2$ 
we see that 
\[
(1+\delta)^{-1} \le A/B \le (1+\delta).
\]
Similarly using $n_1n_3/(n_2n_4) = (r/s)^2$ we have 
\[
(1+\delta)^{-1}\le  \frac{r}{s} \le (1+\delta),
\]
and finally using $n_1n_4/(n_2n_3) = u^2/v^2$ we get that
\[
 (1+\delta)^{-1} \le u/v \le (1+\delta).
\]
In what follows we shall make use of this parametrization and 
the above inequalities for the ratios $A/B$, $r/s$, $u/v$.  One consequence 
of these inequalities is that if $A\neq B$ then $A$ and $B$ are both $\ge 1/\delta$.  Similarly if $r\neq s$ then both $r$ and $s$ are $\ge 1/\delta$ 
and if $u\neq v$ then $u$ and $v$ are both $\ge 1/\delta$.   

\begin{prop}
 \label{fourth} Call any solution to $n_1n_2n_3n_4=\square$ 
where the variables are equal in pairs a {\em diagonal solution}.  
The number of non-diagonal solutions to $n_1 n_2 n_3 n_4 =\square$ 
with $n_j \in [x,x(1+\delta)]$ and $n_j$ square-free is at most 
\[
 80x^2 \delta^3 (1+2 \log x) (1+ 2\delta\log x).
\]
Therefore, with $S$ denoting the number of square-free integers in $[x,x(1+\delta)]$  
\[ 
 {\Bbb E}\Big( \Big(\sum_{k=x}^{x(1+\delta)} X(k) \Big)^4 \Big) 
= 3 S^2 +O(x^2 \delta^3 (1+\delta \log x)\log x ).
\]
\end{prop}
\begin{proof}  Suppose $A$, $B$, $r$, $s$, $u$, $v$ parametrize 
a non-diagonal solution to $n_1n_2n_3n_4=\square$.  Then 
either one of $u$ or $v$ is not $1$, or one of $r$ or $s$ is 
not $1$; for if $u=v=1$ and $r=s=1$ then $n_1=n_2$ and $n_3=n_4$.  
Since these cases are symmetric we will only deal with the 
case when one of $u$ or $v$ is not $1$, and the total 
number of solutions is at most twice the number of solutions 
in this case.  

Suppose then that $u$ or $v$ is not $1$, 
and since $(u,v)=1$ this means that $u\neq v$ and so both $u$ and $v$ are 
$\ge 1/\delta$.  Therefore it 
follows that $Ar \le x\delta(1+\delta)$.  Further either $A\neq B$ or 
$r\neq s$, and so either $A$ or $r$ must be $\ge 1/\delta$.
Now suppose 
$A$ and $r$ are given with $\max(A, r)\ge 1/\delta$ and $Ar \le x\delta(1+\delta)$.  
Since $(1+\delta)^{-1} \le A/B\le (1+\delta)$ it 
follows that there are at most $(1+2A\delta)$ choices 
for $B$.  Similarly since $(1+\delta)^{-1} \le r/s\le 1+\delta$ 
there are at most $1+2r\delta$ choices for $s$.  
Finally since $Aru \in [x,x(1+\delta)]$ there 
are at most $1+x\delta/(Ar) \le 3 x\delta/(Ar)$ choices 
for $u$, and similarly there are at most 
$1+x\delta/(Bs) \le 3 x\delta/(Ar)$ choices for $v$.  Thus 
the total number of such solutions is 
\[
 \le 9x^2 \delta^2 \sum_{{\max(A, r)\ge 1/\delta}\atop {Ar\le x\delta(1+\delta)}} \frac{(1+2A\delta)}{A^2} \frac{(1+2r\delta)}{r^2}.
\]
This may be bounded by 
\begin{align*}
& \le 18x^2 \delta^2 
\sum_{x\delta \ge A\ge 1/\delta} \frac{1+2A\delta}{A^2} 
\sum_{x\delta \ge r\ge 1} \frac{1+2r\delta}{r^2} 
\\
&\le 40 x^2 \delta^3 (1+2\log x) (1+2\delta \log x),
\end{align*}
proving our Proposition.
 \end{proof}
\vskip.1in
When $S$ is of size $x\delta$ (which holds if $x\delta \gg x^{\frac 15}\log x$),  Proposition \ref{fourth} shows that provided $\delta =o(1/\log x)$, 
the fourth moment matches the fourth moment of a Gaussian.  

We now use the ideas of this section to bound the term $\sum_{p \in {\mathcal L}} {\Bbb E}|\Delta_p f|^3$,  
arising in Proposition 2.2.  By the Cauchy-Schwarz inequality we have 
$$ 
{\Bbb E} |\Delta_p f|^3 \le  ({\Bbb E} |\Delta_p(f)|^2 )^{\frac 12} ({\Bbb E} |\Delta_p (f)|^4)^{\frac 12}.
$$ 
As before let ${\mathcal N}(p)$ denote the square-free integers in $(x/p, (x+y)/p]$ 
which are not multiples of $p$.  Then  
$$
{\Bbb E}|\Delta_p f|^2 = 2 \sum_{k \in {\mathcal N}(p)} 1 \le 2 \Big(1+ \frac yp\Big). 
$$ 
Further we have 
$$ 
{\Bbb E} |\Delta_p f|^4 = 8 
\sum_{{k_1, k_2, k_3, k_4 \in {\mathcal N}(p)}  \atop {k_1k_2k_3k_4 =\square}} 1, 
$$ 
and arguing as in Proposition 3.1 we find that this is $\ll (1+y/p)^2$ provided 
$\delta \le 1/\log x$, where $\ll$ means $\le$ up to a constant multiple.  Therefore we conclude that 
$$ 
{\Bbb E}|\Delta_p(f)|^3 \ll 1+\Big(\frac yp\Big)^{\frac 32}.
$$
Using this estimate for primes $z<p\le y$ we find that 
$$ 
\sum_{z<p\le y} {\Bbb E} |\Delta_p f|^3\ll y^{\frac 32} \sum_{z<p} \frac{1}{p^{\frac 32} }
\ll \frac{y^{\frac 32}}{z^{\frac 12}}.
$$
If $p>y$ then ${\Bbb E}|\Delta_p f|^3 =0$ unless there happens to 
be a square-free multiple of $p$ in  $[x,x+y]$ and in this case the 
expectation is $4$.  Such primes $p$ must 
divide $\prod_{x<n\le x+y} n < (x+y)^{y}$ and there are at most $y \log (x+y)/\log y$ possibilities 
for such primes $p$.   We conclude that 
\begin{equation} 
\label{est3}
\sum_{p \in {\mathcal L}} {\Bbb E} |\Delta_p f|^3 
\ll \frac{y^{\frac 32}}{z^{\frac 12}} + y \frac{\log x}{\log y}. 
\end{equation}

\section{Proof of the Theorem}   

\noindent We now estimate $\var(\sum_{p \in {\mathcal L}} T_p)$ where we recall that 
 $T_p$ is defined in \S 2.  This 
quantity equals 
\[
 \sum_{p, q \in {\mathcal L}} \sum_{{k \in {\mathcal N}(p)} \atop {\ell \in {\mathcal N}(p) \backslash \{k\}} }
 \sum_{ {k^{\prime}\in {\mathcal N}(q) } \atop { \ell^{\prime}\in {\mathcal N}(q)\backslash \{{k^{\prime}}\}} }
 \frac{1}{\omega_{\mathcal  L}(\ell p) \omega_{\mathcal L}(\ell^{\prime } q)} 
{\Bbb E}(X(k)X(\ell) X(k^{\prime}) X({\ell^{\prime}})).
\]
Above we allow for the possibility that $p$ equals $q$.  
The expectation above is $1$ exactly when $k\ell k^{\prime}\ell^{\prime}$ is a square 
and zero otherwise.  Thus writing $n_1=kp$, $n_2=\ell p$, $n_3=k^{\prime}q$, $n_4=\ell^{\prime}q$ the quantity we seek is 
\[ 
\sum_{n_1, n_2, n_3, n_4} \frac{\omega_{\mathcal L}((n_1, n_2)) \omega_{\mathcal L}((n_3, n_4))} 
{\omega_{\mathcal L}(n_2) \omega_{\mathcal L}(n_4)} 
\le \sum_{n_1, n_2, n_3, n_4} 1, 
\]
where $n_j \in [x,x(1+\delta)]$, $n_1\neq n_2$, $n_3\neq n_4$, the 
$n_j$ are square-free with $n_1n_2n_3n_4=\square$, and $(n_1,n_2)$ and $(n_3,n_4)$  must 
contain at least one prime factor from ${\mathcal L}$. 
 
We use the parametrization developed in \S 3 to estimate this. 
In the notation used there we find that our quantity above is 
\begin{equation}
 \label{qty}
 \le \sum_{A, B, r, s, u, v} 1.
\end{equation} 
 The sum 
above is over all $A$, $B$, $r$, $s$, $u$, $v$ as 
in our parametrization with the further restraints
that $Aru\neq Asv$ and $Brv \neq Bsu$, and that $A$ and $B$ must each contain 
at least one prime factor from ${\mathcal L}$.  Our goal is to 
show that the above quantity is bounded by 
\begin{equation}
 \label{goal}
O\Big(x^2 \delta^2 (1+\delta\log x) \Big( \frac 1{z} + 
\delta \log x  \Big)\Big).  
\end{equation}
 We will obtain this by first fixing $A$ and $r$ and 
analyzing the restraints on the other variables.   


Suppose first that $A$ and $r$ are chosen with 
$Ar > \delta (x+y)$.  If $u\neq v$ then both $u$ and $v$ must be $\ge 1/\delta$ and 
then we would have $Aru >x+y$.  Thus we must have $u=v$ and since  $(u,v)=1$ 
we have $u=v=1$.   Now $r\neq s$ (else $n_1=A=n_2$) and so we have 
that both $r$ and $s$ are at least $1/\delta$.  Thus we have $A \ll x\delta$ and 
$Ar \in [x,x+y]$.   Given $r$ the condition $(1+\delta)^{-1} \le r/s \le (1+\delta)$
shows that there 
are $\ll r\delta$ choices for $s$.  Similarly the 
inequality  $(1+\delta)^{-1} \le A/B\le (1+\delta)$ shows that given $A$ there are $\ll 1+A\delta$ 
choices for $B$.  Thus in this case our quantity is 
\begin{align*}
& \ll \sum_{A\ll x\delta}   (1+A \delta) \sum_{x/A 
\le r \le (x+y)/A} r\delta 
\ll x^2 \delta^2 \sum_{A\le x}   \frac{(1+A\delta)}{A^2} \\
&\ll x^2\delta^2 \Big(\frac 1{z}+\delta \log x\Big).
\\
\end{align*}
The final estimate follows because $A$ must contain at least 
one prime factor from ${\mathcal L}$, so that $A\ge z$ and hence
 $\sum_{A} 1/A^2 \ll 1/z$.

Now suppose that $Ar <\delta (x+y)$.  Recall that either $r=s=1$ 
or that both $r$ and $s$ are at least $1/\delta$.  We consider these 
cases separately. In the former case, note that 
$B$ has $\ll 1+A\delta$ choices, and $u$ and $v$ have at most $x\delta/A$ 
choices each.   Thus this case contributes 
\[ 
\ll \sum_{A\le \delta (x+y)}   (1+A\delta) x^2\delta^2/A^2 
\ll x^2 \delta^2 \Big( \frac 1{z} + \delta \log x  \Big).
\] 
  Now suppose that we have the 
second case when $r\ge 1/\delta$.  Here there are $\ll 1+A\delta$ 
choices for $B$, and given $r$ there are $\ll r\delta$ choices 
for $s$.  Finally there are $\ll x\delta/(Ar)$ choices for $u$ 
and $\ll x\delta/(Bs) \ll x\delta/(Ar)$ choices for $v$.  Thus 
the contribution here is, 
\begin{align*}
& \ll \sum_{A, r} (1+A\delta) r\delta \frac{x^2 \delta^2}{A^2 r^2}
\\
&\ll x^2 \delta^3 \sum_{A}   (1+A\delta)/A^2 \sum_r 1/r\\
&\ll x^2 \delta^3 \log x \Big( \frac{1}{z} +\delta\log x  \Big).
\end{align*}
Putting all these estimates together gives our bound 
\eqref{goal}. 

Using the bound \eqref{goal}, together with \eqref{est1}, \eqref{est2} and \eqref{est3} we 
conclude that $|{\Bbb E}(\phi(W)) -{\Bbb E}(\phi(Z))|$ is 
$$ 
\ll \Big(\frac{y}{S}\Big)^{\frac 32} \frac{1}{(\log 1/\delta)^{\frac 12}} + \frac{y\log x}{S^{\frac 32}\log y} 
+ \frac{y}{S} (1+\delta \log x)^{\frac 12} \Big( \frac{1}{\log 1/\delta} + \delta\log x\Big)^{\frac 12}. 
$$ 
To deduce the Theorem we combine the above bound with 
the following simple estimate for $|{\Bbb E}(\phi(W)) -{\Bbb E}(\phi(Z))|$.  Since 
$\phi$ is Lipschitz we have $|\phi(t)-\phi(0)| \le |t|$, 
and so 
\begin{align*} 
|{\Bbb E}(\phi(W))-{\Bbb E}(\phi(Z))| &\le |{\Bbb E}(\phi(W)-\phi(0))| + |{\Bbb E}((\phi(Z)-\phi(0))| 
\\
&\le {\Bbb E}(|W|) + {\Bbb E}(|Z|) \le 2. 
\end{align*}

\section{Proof of the Corollary} 

\noindent Let $\nu$ denote a Gaussian distribution with mean $0$ and variance $1$, and 
let $\mu$ denote a probability measure.   We claim that 
\begin{equation} 
\label{KKW}  
{\mathcal K}(\mu,\nu) \le 2\sqrt{{\mathcal W}(\mu,\nu)},
\end{equation} 
and Corollary 1.2 follows as a special case of this estimate.

For any real number $t$, and a parameter $\epsilon >0$ consider 
the function $\Phi^+(\xi;t,\epsilon)$ defined by 
$$\Phi^+(\xi;t,\epsilon)= 
\begin{cases} 
\epsilon &\text{if  } \xi\in (-\infty,t)\\ 
t+\epsilon -\xi &\text{if } \xi \in [t,t+\epsilon]\\ 
0&\text{if } \xi > t+\epsilon.
\end{cases}
$$
Note that $\Phi^+(\xi;t,\epsilon)$ is Lipschitz, and moreover 
$\Phi^+(\xi;t,\epsilon) \ge \epsilon \chi_{(-\infty,t)}(\xi)$.  Therefore 
\begin{align*}
\int_{-\infty}^t d\mu & \le \frac{1}{\epsilon} \int_{-\infty}^{t} \Phi^+(\cdot;t,\epsilon) d\mu
\le \frac{1}{\epsilon} \int_{-\infty}^{t} \Phi^+(\cdot; t,\epsilon) d\nu + \frac{{\mathcal W}(\mu,\nu)}{\epsilon} \\ 
&\le \int_{-\infty}^{t} d\nu + \epsilon + \frac{{\mathcal W}(\mu,\nu)}{\epsilon}.
\end{align*} 
Choosing $\epsilon =\sqrt{{\mathcal W}(\mu,\nu)}$ we obtain that 
$$ 
\int_{-\infty}^{t} d\mu \le \int_{-\infty}^{t} d\nu  + 2\sqrt{{\mathcal W}(\mu,\nu)}. 
$$ 

An analogous argument, using a similar Lipschitz minorant of the characteristic function 
of $(-\infty,t)$, gives that 
$$ 
\int_{-\infty}^{t} d\mu \ge \int_{-\infty}^{t} d\nu - 2\sqrt{{\mathcal W}(\mu,\nu)},
$$ 
and so (\ref{KKW}) follows.

 \vskip .1 in 
 
 \noindent {\bf Acknowledgements.}  We are happy to thank Persi Diaconis for 
 facilitating this collaboration, and many valuable discussions.   We are also grateful to 
 Adam Harper, Zeev Rudnick and the referee for some helpful comments.

\end{document}